\documentclass[12pt,a4paper,twoside]{article}

\title{\sc Estimates for Deviations from Exact Solutions of Maxwell's Initial Boundary Value Problem}
\def\shorttitle{Deviations from Exact Solutions of Maxwell's Equations}
\def\pauthor{Dirk Pauly, Sergey Repin, Tuomo Rossi}

\def\mylabelonoff{off}
\def\allowdisbrk{no}

\usepackage{a4,exscale,ifthen,amsfonts,amssymb,amsmath,amscd,graphicx}
\usepackage[english]{babel}
\usepackage[mathscr]{eucal}

\author{{\sf\pauthor}}
\numberwithin{equation}{section}

\newenvironment{acknow}{{\vspace*{1cm}\noindent\bf Acknowledgements }}{}
\newcommand{\bewboxw}{\mbox{}\hfill $\square$ \\}

\newenvironment{proof}{{\noindent\bf Proof }}{\bewboxw}
\newenvironment{proofof}[1]{{\noindent\bf Proof of #1 }}{\bewboxw}

\newcommand{\keywords}[1]{{\noindent\bf Key Words }#1}
\newcommand{\amsclass}[1]{{\noindent\bf AMS MSC-Classifications }#1}

\ifthenelse{\equal{\mylabelonoff}{on}}
{\newcommand{\mylabel}[1]{\label{#1}\fbox{{\rm #1}}}}{\newcommand{\mylabel}[1]{\label{#1}\makebox[0mm][]{}}}
\ifthenelse{\equal{\allowdisbrk}{yes}}
{\allowdisplaybreaks}{}

\newcommand{\paper}[7]{\bibitem{#1} #2, `#3', {\it #4}, #5, (#6), #7.}

\newcommand{\papersub}[4]{\bibitem{#1} #2, `#3', (#4), submitted.}
\newcommand{\book}[6]{\bibitem{#1} #2, {\it #3}, #4, #5, (#6).}

\newcommand{\repjyu}[8]{\bibitem{#1} #2, `#3', {\it Reports of the Department of Mathematical Information Technology}, 
University of Jyv\"askyl\"a, Series #4. Scientific Computing, No. #4. #5/#6, ISBN #7, ISSN #8.}

\usepackage[mathscr]{eucal}

\newcommand{\ol}{\overline}

\newcommand{\ub}{\underbrace}
\newcommand{\nz}{\mathbb{N}}

\newcommand{\rz}{\mathbb{R}}
\newcommand{\rzp}{\rz_{+}}

\newcommand{\rN}{\rz^N}
\newcommand{\rt}{\rz^3}

\DeclareMathOperator{\p}{\partial}
\newcommand{\pt}{\p_{t}}

\DeclareMathOperator{\curl}{curl}

\DeclareMathOperator{\iu}{i}

\newcommand{\mop}{M_{\Lambda}}
\newcommand{\eps}{\varepsilon}

\newcommand{\om}{\Omega}

\newcommand{\omt}{\om_t}

\def\e{\mathrm{e}}

\newcommand{\Et}{\tilde{E}}
\newcommand{\Ett}{\tilde{E}_{t}}

\newcommand{\et}{e_{t}}

\newcommand{\calB}{\mathcal{B}}
\newcommand{\calE}{\mathcal{E}}

\newcommand{\Hi}{\mathcal{H}(\Omega)}
\newcommand{\domdef}{\mathcal{D}}
\newcommand{\tmat}[4]{\begin{bmatrix}#1&#2\\#3&#4\end{bmatrix}}
\newcommand{\nrom}[1]{n_{#1,\rho}}
\newcommand{\Nrom}[1]{N_{#1,\rho}}
\newcommand{\noom}[1]{n_{#1,1}}

\newcommand{\ntrom}[1]{\tilde{n}_{#1,\rho}}
\newcommand{\Nrgom}[1]{N_{#1,\rho,\gamma}}
\newcommand{\Nroom}[1]{N_{#1,\rho,1}}
\newcommand{\Ntrgom}[1]{\tilde{N}_{#1,\rho,\gamma}}
\newcommand{\Ntroom}[1]{\tilde{N}_{#1,\rho,1}}

\def\Set#1#2{\left\{#1\,\mid\,#2\right\}}

\DeclareMathOperator{\Lebesgue}{\mathsf{L}}
\newcommand{\Lgen}[2]{\Lebesgue^{#1}_{#2}}
\def\Lo{\Lgen{1}{}}

\def\Lt{\Lgen{2}{}}

\def\Ltom{\Lt(\om)}

\DeclareMathOperator{\Sobolev}{\mathsf{H}}
\newcommand{\Hgen}[3]{\overset{#3}{\Sobolev}{}^{#1}_{#2}}

\def\Ht{\Hgen{2}{}{}}
\def\Hm{\Hgen{m}{}{}}

\def\Htom{\Ht(\om)}
\def\Hmom{\Hm(\om)}

\DeclareMathOperator{\Cont}{\mathsf{C}}
\newcommand{\Cgen}[2]{\overset{#2}{\Cont}{}^{#1}}

\def\Ciz{\Cgen{\infty}{\circ}}

\def\Cizom{\Ciz(\om)}
\def\Cz{\Cgen{0}{}}
\def\Co{\Cgen{1}{}}
\def\Ct{\Cgen{2}{}}

\def\Cl{\Cgen{\ell}{}}
\def\Clme{\Cgen{\ell-1}{}}
\def\Cop{\Cgen{1}{}_{p}}
\def\Ctp{\Cgen{2}{}_{p}}
\def\Clp{\Cgen{\ell}{}_{p}}

\newcommand{\Hcgenom}[3]{\Sobolev_{#1}(\curl^{#2}_{#3},\om)}
\newcommand{\Hcom}[1]{\Hcgenom{#1}{}{}}
\newcommand{\Hccom}[1]{\Hcgenom{#1}{\circ}{}}

\newcommand{\Hom}[1]{\Hgen{}{#1}{}(\om)}

\newcommand{\normdst}{\hspace{-0.4ex}}
\newcommand{\scp}[2]{\left\langle#1,#2\right\rangle}
\newcommand{\scpom}[2]{\scp{#1}{#2}_{\om}}
\newcommand{\scpomt}[2]{\scp{#1}{#2}_{\omt}}

\newcommand{\norm}[1]{\left|\normdst\left|#1\right|\normdst\right|}

\newcommand{\normmuom}[1]{\norm{#1}_{\mu,\om}}
\newcommand{\normmumoom}[1]{\norm{#1}_{\mu^{-1},\om}}
\newcommand{\normepsom}[1]{\norm{#1}_{\eps,\om}}
\newcommand{\normepsmoom}[1]{\norm{#1}_{\eps^{-1},\om}}
\newcommand{\normemrom}[1]{\norm{#1}_{\eps,\mu^{-1},\rho,\om}}

\newcommand{\normmuomt}[1]{\norm{#1}_{\mu,\om_{t}}}
\newcommand{\normmumoomt}[1]{\norm{#1}_{\mu^{-1},\omt}}
\newcommand{\normepsomt}[1]{\norm{#1}_{\eps,\omt}}
\newcommand{\normepsmoomt}[1]{\norm{#1}_{\eps^{-1},\omt}}
\newcommand{\normemromt}[1]{\norm{#1}_{\eps,\mu^{-1},\rho,\omt}}

\newtheorem{lem}{Lemma}%[section]

\newtheorem{theo}[lem]{Theorem}
\newtheorem{cor}[lem]{Corollary}
\newtheorem{rem}[lem]{Remark}

\newcommand{\addressesdirksergeytuomo}{
\footnotesize
\vspace*{1cm}
\begin{center}
\begin{tabular}{lll}
{\sf\small Dirk Pauly} \vspace*{1mm}\\
Universit\"at Duisburg-Essen \\
Campus Essen \\
Fakult\"at f\"ur Mathematik \\
Universit\"atsstr. 2 \\
45117 Essen \\
Germany \\
e-mail: {\tt dirk.pauly@uni-due.de} \\
\\
{\sf\small Sergey Repin} \vspace*{1mm}\\
V.A. Steklov Mathematical Institute & and & University of Jyv\"askyl\"a \\
St. Petersburg Branch & & Faculty of Information Technology \\
 & & Department of \\
 & & Mathematical Information Technology \\
Fontanka 27 & & P.O. Box 35 (Agora) \\
191011 St. Petersburg & & FI-40014 Jyv\"askyl\"a \\
Russia & & Finland \\
e-mail: {\tt repin@pdmi.ras.ru} & & e-mail: {\tt serepin@jyu.fi} \\
\\
{\sf\small Tuomo Rossi} \vspace*{1mm}\\
University of Jyv\"askyl\"a \\
Faculty of Information Technology \\
Department of \\
Mathematical Information Technology \\
P.O. Box 35 (Agora) \\
FI-40014 Jyv\"askyl\"a \\
Finland \\
e-mail: {\tt tuomo.rossi@jyu.fi}
\end{tabular}
\end{center}}

\begin{document}

\maketitle{}

\begin{abstract}
In this paper, we consider an initial boundary value problem
for Maxwell's equations. For this hyperbolic type problem,
we derive guaranteed and computable upper bounds for the difference
between the exact solution and any pair of vector fields in the
space-time cylinder that belongs to the corresponding
admissible energy class. For this purpose,
we use a method suggested in \cite{ReWave} for the wave equation.\\
\\
\amsclass{65N15, 35L15, 78M30}
\end{abstract}

\tableofcontents

\section{Introduction}

In this paper, we derive computable upper bounds for the
distance between the exact solution $(E,H)$
of an initial boundary value problem for Maxwell's equations,
which have in their second order form a hyperbolic nature,
and any pair of vector fields $(\Et,\tilde{H})$ belonging to the
admissible energy class of the considered problem.
As our techniques rely on second order methods
and the Maxwell system decouples
in its second order version for the electric field $E$
and the magnetic field $H$, we focus on $E$ in our analysis. 
The vector field $\Et$ can be considered as an approximation of $E$ 
computed with the help of a numerical method.  
In other words, 
we deduce nonnegative functionals $\calB$ 
(also called error majorants or upper bounds) 
that depend only on $\Et$ and known data 
(coefficients, domain, right hand side and boundary data)
and satisfy the following properties:
\begin{enumerate}
\item $\calE(E-\Et)\leq\calB(\Et)$ for all admissible $\Et$.
\item $\calB(\Et)=0$ iff $\Et=E$.
\item $\calB(\Et)\to0$ if $\calE(E-\Et)\to0$.
\end{enumerate}
Here, $\calE$ is a suitable error measure
of the energy of the system defined on a space-time cylinder
(e.g., a $\Lt$-energy norm).

Such Functionals $\calB$ provide an explicit verification 
of the accuracy of approximations. 
Indeed, we see that $\calB(\Et)$ is small then
$\Et$ belongs to a certain neighborhood of the exact solution. 
Moreover, $\calB$ vanishes only at the exact solution $E$.
The third property shows that the majorant $\calB$
possesses the continuity property with respect to all sequences
converging in the topology induced by the energy norm $\calE$.

Estimates of such a type (often called {\em functional a posteriori estimates}) 
can be derived by at least two methods. 
The first method is based on variational techniques and 
applicable for problems that admit a variational statement. 
By this method a posteriori error estimates were derived in \cite{ReNonlinvar,ReVar} 
and many other publications (see \cite{NeRe} for a systematic overview).
Another method is based upon the analysis of the
integral identity (variational formulation) 
that defines the corresponding generalized solution. 
This method was suggested in \cite{ReTwoSided}, 
where it was also shown that for linear elliptic equations both methods
(variational and nonvariational) lead to the same estimates.
Later the nonvariational method was also applied to nonlinear elliptic problems
and to certain classes of nonlinear problems in continuum mechanics
(e.g., for variational inequalities \cite{BiFuRe,FuRe,Re2007})
and to initial boundary value problems associated with parabolic type
equations \cite{ReLincei}.
A consequent exposition of the 'nonvariational' a posteriori error estimation method
is presented in the book \cite{ReGruyter}.
Analogous estimates have been derived for elliptic problems 
in exterior domains as well \cite{PaReEll}.

In this paper, we are concerned with an initial boundary value problem 
for Maxwell's equations. For the stationary version of this problem, 
functional a posteriori estimates have been derived earlier 
in \cite{PaReMaxbd} for bounded domains.
However, the hyperbolic Maxwell problem
essentially differs from the stationary case 
and the estimates are derived by a new technique. 
The derivation method is also based on the analysis
of a basic integral relation 
but uses a rather different {\em modus operandi}. 
The reason for this lies in the specific properties of the respective differential
operator involving second order time and spatial derivatives with opposite signs.  
We overcome the difficulties arising due to this fact
with the help of a method suggested in \cite{ReWave} for the wave equation,
which is closely related,
and deduce computable upper bounds for the distance to the exact solution
measured in a canonical $\Lt$-energy norm.

Our main results are presented in Section \ref{secfirstest} 
by Theorems \ref{theoone} and \ref{theotwo}, 
which provide computable and guaranteed majorants 
for the error measures \eqref{errone} and \eqref{errtwo}
of the electric field $E$.
These first (and simplest) majorants are derived
under stronger assumptions on the approximation.
They can be used if the approximation $\Et$ possesses
extra regularity, which sometimes may be difficult to guarantee in many numerical schemes.
In Section \ref{secsecondest} we prove corresponding results
under weaker assumption on the approximation $\Et$,
which are free of these drawbacks, but have a more complicated structure.
Finally, in section \ref{bothfields} we estimate the error 
for the approximation of the magnetic field $H$ as well
and thus the error of the approximation of the full solution $(E,H)$. 

We note that the respective functionals generate new variational problems,
where exact lower bounds vanish and are attained only on the exact solution. 
In applied analysis, the functionals can be used 
for a posteriori control of errors of approximate solutions 
obtained by various numerical methods. 

\section{Basic problem}

Let $\Omega$ be a domain\footnote{, i.e., a connected open set,} 
in $\rt$ with Lipschitz continuous boundary
$\Gamma:=\partial\Omega$ and corresponding outward unit normal vector by $\nu$.
Furthermore, let $T>0$, $I:=(0,T)$ as well as $\Omega_{t}:=(0,t)\times\Omega$
and $\Gamma_{t}:=(0,t)\times\Gamma$ for all $t>0$
the space-time cylinder and cylinder barrel, respectively.
We consider the classical initial boundary value problem for Maxwell's equation:
Find vector fields $E$ and $H$ (electric and magnetic field), such that
\begin{align}
\pt E-\eps^{-1}\curl H&=F&&\text{in }\Omega_{T},\mylabel{maxeqone}\\
\pt H+\mu^{-1}\curl E&=G&&\text{in }\Omega_{T},\mylabel{maxeqtwo}\\
\nu\times E|_{\Gamma}&=0&&\text{on }\Gamma_{T},\mylabel{maxeqbc}\\
E(0)=E(0,\,\cdot\,)&=E_{0}&&\text{in }\Omega,\mylabel{maxeqicone}\\
H(0)=H(0,\,\cdot\,)&=H_{0}&&\text{in }\Omega.\mylabel{maxeqictwo}
\end{align}
Here $\eps$ and $\mu$ denote time-independent, real, symmetric and positive definite matrices 
with measurable, bounded coefficients that describe properties of the
media (dielectricity and permeability, respectively).
For the sake of brevity, matrices (matrix-valued functions) 
with such properties are called 'admissible'.
We note that the corresponding inverse matrices are admissible as well. 

\begin{rem}
\mylabel{remdomain}
The underlying domain $\Omega$ may be bounded or unbounded.
Contrary to the stationary cases, i.e., static or time-harmonic equations,
the Sobolev spaces used for the solution theory of the Cauchy problem 
do not differ whether the domain is bounded or not. 
For instance, in exterior domains one has to work with polynomially weighted Sobolev spaces
what naturally would lead to weighted error estimates as well.  
See \cite{kuhnpaulyreg,paulytimeharm,paulystatic,paulydeco,paulyasym} 
for a detailed description.
\end{rem}

By $\Ltom$ we denote the usual scalar $\Lt$-Hilbert space
of square integrable functions on $\Omega$
and by $\Hmom$, $m\in\nz$, the usual Sobolev spaces.
$\Hom{}$ denotes the Hilbert space of real-valued $\Lt$-vector fields, i.e., $\Lt(\Omega,\rt)$.
For the sake of simplicity, we restrict our analysis to the case of real-valued
functions and vector fields. The generalization to complex-valued spaces
is straight forward. Moreover, we define
$$\Hcom{}:=\Set{\Phi\in\Hom{}}{\curl\Phi\in\Hom{}},\quad
\Hccom{}:=\overline{\Cizom},$$
where the closure is taken in the natural norm of $\Hcom{}$.
The homogeneous tangential boundary condition
\eqref{maxeqbc} is generalized in $\Hccom{}$ by Gau{\ss}' theorem.
Equipped with their natural scalar products all these spaces are Hilbert spaces.

To formulate and obtain a proper Hilbert space solution theory for the latter Cauchy problem,
we need some more suitable Hilbert spaces.
We set
$$\Hi:=\Hom{}\times\Hom{}$$
as a set and equip this space with the weighted scalar product
$$\scp{(E,H)}{(\Phi,\Psi)}_{\Hi}:=\scp{\Lambda(E,H)}{(\Phi,\Psi)}_{\Hom{}\times\Hom{}}
=\scp{\eps E}{\Phi}_{\Hom{}}+\scp{\mu H}{\Psi}_{\Hom{}},$$
where
$$\Lambda:=\tmat{\eps}{0}{0}{\mu}.$$
For the sake of a short notation, we will write for domains $\Xi\subset\rN$
$$\norm{\,\cdot\,}_{\Xi}:=\norm{\,\cdot\,}_{\Lt(\Xi,\rz^{\ell})},\quad
\scp{\,\cdot\,}{\,\cdot\,}_{\Xi}:=\scp{\,\cdot\,}{\,\cdot\,}_{\Lt(\Xi,\rz^{\ell})}$$
and for suitable matrices $A$
$$\norm{\,\cdot\,}_{A,\Xi}:=\norm{A^{1/2}\,\cdot\,}_{\Xi}=\scp{A\,\cdot\,}{\,\cdot\,}_{\Xi}^{1/2}.$$
Furthermore, we introduce the linear operator
$$\mop:\domdef(\mop)\subset\Hi\to\Hi,\quad(\Phi,\Psi)\mapsto\iu\Lambda^{-1}M(\Phi,\Psi)$$
putting
$$\domdef(\mop):=\Hccom{}\times\Hcom{},\quad M:=\tmat{0}{-\curl}{\curl}{0}.$$
Then, a solution of the Cauchy problem \eqref{maxeqone}-\eqref{maxeqictwo}
is to be understood as a solution of the Cauchy problem
\begin{align}
(\pt-\iu\mop)(E,H)&=(F,G),\mylabel{maxeq}\\
(E,H)(0)&=(E_{0},H_{0}).\mylabel{maxeqic}
\end{align}

Utilizing a slight and obvious modification (variation of constant formula) 
of \cite[Theorem 8.5]{leisbook},
the Cauchy problem \eqref{maxeq}-\eqref{maxeqic}
has unique solution for all $T$ (we may also replace the interval $I$ by $\rz$)
by spectral theory since $\mop$ is self-adjoint. The spectral theorem suggests
$$(E,H)(t)=\exp(\iu t\mop)(E_{0},H_{0})+\int_{0}^{t}\exp(\iu(t-s)\mop)(F,G)(s)\,ds,\quad 
t\in\ol{I}$$
as solution. We get:

\begin{theo}
\mylabel{existtheo}
Let $(F,G)\in\Lo(I,\Hi)$ and $(E_{0},H_{0})\in\Hi$.
Then, the Cauchy problem \eqref{maxeq}-\eqref{maxeqic} is uniquely solvable in
\begin{itemize}
\item[\bf(i)] $\Cz(\ol{I},\Hi)$;
\item[\bf(ii)] $\Cz(\ol{I},\domdef(\mop))\cap\Co(\ol{I},\Hi)$, if additionally\\
$(F,G)\in\Lo(I,\domdef(\mop))\cap\Cz(\ol{I},\Hi)$
and $(E_{0},H_{0})\in \domdef(\mop)$;
\item[\bf(iii)] $\Cz(\ol{I},\domdef(\mop^2))\cap\Co(\ol{I},\domdef(\mop))\cap\Ct(\ol{I},\Hi)$, 
if additionally\\
$(F,G)\in\Lo(I,\domdef(\mop^2))\cap\Cz(\ol{I},\domdef(\mop))\cap\Co(\ol{I},\Hi)$
and $(E_{0},H_{0})\in \domdef(\mop^2)$.
\end{itemize}
\end{theo}

Here, $(E,H)\in\domdef(\mop^2)$, if and only if 
$$(E,H),(\eps^{-1}\curl H,\mu^{-1}\curl E)\in\domdef(\mop)=\Hccom{}\times\Hcom{}.$$

\begin{rem}
\mylabel{existtheorem}
\begin{itemize}
\item[\bf(i)] Theorem \ref{existtheo} holds if we replace the spaces $\Cl$ 
by spaces of vector fields having such regularity only piecewise, i.e., $\Clp$,
where $\Phi\in\Clp$, if and only if $\Phi\in\Clme$ and $\Phi$ is piecewise $\Cl$.  
\item[\bf(ii)] To obtain the second order regularity in Theorem \ref{existtheo} (iii)
and in view of numerical applications
it is sufficient to assume that $(E_{0},H_{0})$ has $\Htom$-components 
and that $(F,G)$ has $\Ct(\ol{\Omega}_{T})$-components with bounded derivatives.
\end{itemize}
\end{rem}

If $(E,H)$ admits the second order regularity of Theorem \ref{existtheo} (iii)
then we can apply $\pt+\iu\mop$ to \eqref{maxeq} and obtain
$$(\pt^2+\mop^2)(E,H)=(\tilde F,\tilde G),$$
where $(\tilde F,\tilde G):=(\pt+\iu\mop)(F,G)$.
Equivalently, we have
$$(\pt\Lambda\pt-M\Lambda^{-1}M)(E,H)=\Lambda(\tilde F,\tilde G).$$
Since 
$$-M\Lambda^{-1}M=\tmat{\curl\mu^{-1}\curl}{0}{0}{\curl\eps^{-1}\curl}$$
the latter equation decouples for the electric field $E$ and magnetic field $H$. 

In this paper, we intend to discuss the second order system
for the electric field $E$, which reads in classical terms
\begin{align}
(\pt\eps\pt+\curl\mu^{-1}\curl)E&=K:=\eps\tilde F&&\text{in }\Omega_{T},\mylabel{maxeqcurlcurl}\\
\nu\times E|_{\Gamma}&=0&&\text{on }\Gamma_{T},\mylabel{maxeqbctwo}\\
E(0)&=E_{0}&&\text{in }\Omega,\mylabel{maxeqicE}\\
\pt E(0)&=E_{0}':=\eps^{-1}\curl H_{0}+F(0)&&\text{in }\Omega.\mylabel{maxeqicptE}
\end{align}
Moreover, we assume throughout this paper that the second order regularity 
of Theorem \ref{existtheo} (iii) holds.

\begin{rem}
\mylabel{secondtofirstorder}
A solution of the second order problem \eqref{maxeqcurlcurl}-\eqref{maxeqicptE}
provides also a solution of the original first order problem 
\eqref{maxeqone}-\eqref{maxeqictwo}. 
In particular, under proper regularity assumptions on the data
the system \eqref{maxeqcurlcurl}-\eqref{maxeqicptE} 
is uniquely solvable as well. 
To show this it sufficies to set
$$H(t):=\int_{0}^{t}(G(s)-\mu^{-1}\curl E(s))ds+H_{0}.$$
Then, \eqref{maxeqbc} and \eqref{maxeqicone} hold and 
\eqref{maxeqtwo} and \eqref{maxeqictwo} follow directly. 
Furthermore, to prove \eqref{maxeqone} we use 
\eqref{maxeqcurlcurl} and the above definition of $H$ and obtain
\begin{align*}
\pt E(t)&=\int_{0}^{t}\p_{s}^2E(s)ds+E_{0}'
=-\eps^{-1}\curl\int_{0}^{t}\mu^{-1}\curl E(s)ds+\int_{0}^{t}\tilde{F}(s)ds+E_{0}'\\
&=\eps^{-1}\curl H(t)+\int_{0}^{t}(\ub{\tilde{F}(s)-\eps^{-1}\curl G(s)}_{=\p_{s}F(s)})ds
+\ub{E_{0}'-\eps^{-1}\curl H_{0}}_{=F(0)}.
\end{align*}
Hence, our further analysis is based on \eqref{maxeqcurlcurl}-\eqref{maxeqicptE}. 
\end{rem}

\section{First form of the deviation majorant}\mylabel{secfirstest}

Let $\Et$ be an approximation of $E$.
In this section, we assume that

\begin{align}
\mylabel{secordreget}
\Et&\in\Cop(\ol{I},\Hcom{})\cap\Ctp(\ol{I},\Hom{}).
\end{align}

Our goal is to find a computable upper bound of the error 
$$e:=E-\Et$$
associated with $\Et$.
For all $t\in\ol{I}$ and $\rho\in\rzp$ we define two nonnegative quantities
\begin{align}
\nrom{\Phi}(t):=\normemrom{\Phi}^2(t)&:=\normepsom{\pt\Phi}^2(t)
+\rho\normmumoom{\curl\Phi}^2(t),\mylabel{errone}\\
\Nrom{\Phi}(t):=\normemromt{\Phi}^2&:=\normepsomt{\pt\Phi}^2
+\rho\normmumoomt{\curl\Phi}^2,\mylabel{errtwo}
\end{align}
which generate natural energy norms for the accuracy evaluation.
We note that by Fubini's theorem $\int_{\omt}=\int_{0}^{t}\int_{\om}$ and thus
$$\nrom{\Phi}=\Nrom{\Phi}',\quad\Nrom{\Phi}(t)=\int_{0}^{t}\nrom{\Phi}(s)ds.$$

\begin{theo}
\mylabel{theoone}
Let $\rho\in(0,1)$ and $\Et$ be an approximation satisfying \eqref{secordreget}.
Moreover, let $\pt e\in\Hccom{}$ for all $t\in I$. Then, for all $t\in\ol{I}$
\begin{align}
\mylabel{theoonemain}
\nrom{e}(t)\leq\inf_{Y,\gamma}b_{\Et,\rho;Y,\gamma}(t),\quad
\Nrom{e}(t)\leq\inf_{Y,\gamma}B_{\Et,\rho;Y,\gamma}(t),
\end{align}
where
\begin{align*}
b_{\Et,\rho;Y,\gamma}(t)&
:=\gamma\e^{\gamma t}\int_{0}^{t}\e^{-\gamma s}f_{\Et,Y,\gamma,\rho}(s)ds
+f_{\Et,Y,\gamma,\rho}(t),\\
B_{\Et,\rho;Y,\gamma}(t)&
:=\e^{\gamma t}\int_{0}^{t}\e^{-\gamma s}f_{\Et,Y,\gamma,\rho}(s)ds
\end{align*}
and the infima are taken over $\gamma\in\rzp$ and $Y\in\Cop(\ol{I},\Hcom{})$.
Here,
\begin{align*}
f_{\Et,Y,\gamma,\rho}&:=g_{\Et,Y,\gamma,\rho}+z_{\Et,Y},\\
g_{\Et,Y,\gamma,\rho}(t)&:=\gamma^{-1}\normepsmoomt{\hat{K}_{\Et,Y}}^2
+(\gamma\rho)^{-1}\normmuomt{\pt \tilde{K}_{\Et,Y}}^2
+(1-\rho)^{-1}\normmuom{\tilde{K}_{\Et,Y}}^2(t),\\
z_{\Et,Y}&:=\noom{e}(0)+2\scpom{\tilde{K}_{\Et,Y}}{\curl e}(0),
\end{align*}
where
\begin{align*}
\hat{K}_{\Et,Y}&:=\hat{K}_{\pt^2\Et,\curl Y}:=\eps\pt^2\Et+\curl Y-K,\\
\tilde{K}_{\Et,Y}&:=\tilde{K}_{\curl\Et,Y}:=\mu^{-1}\curl\Et-Y.
\end{align*}
\end{theo}

\begin{rem}
\mylabel{theooneremone}
We outline that the functionals $f_{\Et,Y,\gamma,\rho}$, 
$b_{\Et,\rho;Y,\gamma}$ and $B_{\Et,\rho;Y,\gamma}$
depend only on known data, the approximation $\Et$,
the free variable $Y$ and the free parameters $\rho$, $\gamma$,
and do not involve the unknown exact solution $E$.
Thus, these quantities are explicitly computable
once the approximate solution $\Et$ has been constructed.
We note that the 'zero term' $z_{\Et,Y}$ represents the error
in the initial conditions. In particular, 
$z_{\Et,Y}$ mainly consists of
first order derivatives of the initial error 
$e(0)=E_{0}-\Et(0)$.
Furthermore, $Y$ may even be chosen from the larger space 
$$\Cop(\ol{I},\Hom{})\cap\Lt(I,\Hcom{}).$$
\end{rem}

\begin{rem}
\mylabel{theooneremtwo}
The absolute value of the zero term
$$z_{\Et,Y}=\normepsom{\pt e}^2(0)+\normmumoom{\curl e}^2(0)
+2\scpom{\tilde{K}_{\Et,Y}}{\curl e}(0)$$ 
may be estimated from above by the quantities
\begin{align*}
\tilde{z}_{\Et,Y}&:=\noom{e}(0)+2|\scpom{\tilde{K}_{\Et,Y}}{\curl e}(0)|\\
&\,\,=\normepsom{\pt e}^2(0)+\normmumoom{\curl e}^2(0)
+2|\scpom{\tilde{K}_{\Et,Y}}{\curl e}|(0)\\
\leq\hat{z}_{\Et,Y}&:=\normepsom{\pt e}^2(0)+2\normmumoom{\curl e}^2(0)
+\normmuom{\tilde{K}_{\Et,Y}}^2(0),
\end{align*}
which are nonnegative and easily computable. Furthermore,
\begin{itemize}
\item[\bf(i)] $z_{\Et,Y}=0$, 
if $\pt e(0)=0$ and $\curl e(0)=0$.
\item[\bf(ii)] $\tilde{z}_{\Et,Y}=0$, 
if any only if $\pt e(0)=0$ and $\curl e(0)=0$.
\item[\bf(iii)] $\hat{z}_{\Et,Y}=0$, 
if any only if $\pt e(0)=0$ and $\curl e(0)=0$ and $\mu Y(0)=\curl\Et(0)$.
\item[\bf(iv)] $g_{\Et,Y,\gamma,\rho}=0$, 
if and only if $\mu Y=\curl\Et$ and $\curl Y=K-\eps\pt^2\Et$.
\end{itemize}
Therefore, choosing the functional $f_{\Et,Y,\gamma,\rho}$ 
with $\tilde{z}_{\Et,Y}$ or $\hat{z}_{\Et,Y}$ we see that
for all $t\in\ol{I}$ the functional $b_{\Et,\rho;Y,\gamma}(t)$
vanishes, if and only if
\begin{align}
\mylabel{bvanish}
\pt\Et(0)=E_{0}',\quad\curl\Et(0)=\curl E_{0},\quad\mu Y=\curl\Et,\quad\curl Y=K-\eps\pt^2\Et.
\end{align}
Thus, $\pt e$ and $\curl e$ vanish, if and only if $\nrom{e}=0$,
which is implied by $b_{\Et,\rho;Y,\gamma}=0$.
The latter condition is equivalent to \eqref{bvanish}.
The same holds for the energy norm $\Nrom{e}$ and the functional $B_{\Et,\rho;Y,\gamma}$.
\end{rem}

\begin{proofof}{Theorem \ref{theoone}}
We start with deriving first order ordinary differential inequalities, 
which then lead to the estimates by Gronwall's lemma (see appendix). 
Since $\pt e$ belongs to $\Hccom{}$, we have
\begin{align*}
\pt\noom{e}(t)&=2\scpom{\eps\pt^2e}{\pt e}(t)+2\scpom{\mu^{-1}\curl e}{\curl\pt e}(t)\\
&=2\scpom{K-\eps\pt^2\Et}{\pt e}(t)-2\scpom{\mu^{-1}\curl\Et-Y+Y}{\curl\pt e}(t)\\
&=-2\scpom{\hat{K}_{\Et,Y}}{\pt e}(t)
-2\pt\scpom{\tilde{K}_{\Et,Y}}{\curl e}(t)+2\scpom{\pt \tilde{K}_{\Et,Y}}{\curl e}(t).
\end{align*}
Thus, by integration
\begin{align}
\mylabel{estnone}
\noom{e}(t)=z_{\Et,Y}-2\big(\ub{\scpom{\tilde{K}_{\Et,Y}}{\curl e}}_{=:S_{1}}(t)
-\ub{\scpomt{\pt \tilde{K}_{\Et,Y}}{\curl e}}_{=:S_{2}(t)}
+\ub{\scpomt{\hat{K}_{\Et,Y}}{\pt e}}_{=:S_{3}(t)}\big).
\end{align}
We estimate the scalar products $S_{\ell}$ by 
\begin{align}
\begin{split}
2|S_{1}(t)|&\leq\alpha\normmumoom{\curl e}^2(t)+\alpha^{-1}\normmuom{\tilde{K}_{\Et,Y}}^2(t),\\
2|S_{2}(t)|&\leq\beta\normmumoomt{\curl e}^2+\beta^{-1}\normmuomt{\pt \tilde{K}_{\Et,Y}}^2,\\
2|S_{3}(t)|&\leq\gamma\normepsomt{\pt e}^2+\gamma^{-1}\normepsmoomt{\hat{K}_{\Et,Y}}^2,
\end{split}\mylabel{estSl}
\end{align}
where $\alpha,\beta,\gamma\in\rzp$ are arbitrary. For $\rho\in(0,1)$ 
we choose $\alpha:=1-\rho\in(0,1)$. Then, for arbiratry $\gamma\in\rzp$ 
we put $\beta:=\gamma\rho\in\rzp$. Inserting \eqref{estSl} into \eqref{estnone},
we achieve
$$\nrom{e}\leq\gamma\Nrom{e}+f_{\Et,Y,\gamma,\rho},$$
which may be written in two ways
$$\nrom{e}(t)\leq\gamma\int_{0}^{t}\nrom{e}(s)ds+f_{\Et,Y,\gamma,\rho}(t),\quad
\Nrom{e}'(t)\leq\gamma\Nrom{e}(t)+f_{\Et,Y,\gamma,\rho}(t).$$
Gronwall's inequalities, 
i.e., \eqref{gronwalloneesttwo} and \eqref{gronwalltwoesttwo}, 
complete the proof.
\end{proofof}

Since $\pt e=0$ implies $e(t)=e(0)=E_{0}-\Et(0)$ constantly for all $t$ we obtain:

\begin{theo}
\mylabel{theotwo}
Let an approximation $\Et$ and $Y$ as in Theorem \ref{theoone} be given.
Then, the following two statements are equivalent:
\begin{itemize}
\item[\bf(i)] $\Et(0)=E_{0}$ and $b_{\Et,\rho;Y,\gamma}=0$.
\item[\bf(ii)] $\Et=E$ and $\mu Y=\curl E$.
\end{itemize}
In words: Let the approximation $\Et$ satisfy the first initial condition $\Et(0)=E_{0}$ exactly.
Then, the functional $b_{\Et,\rho;Y,\gamma}$ vanishes if and only if 
the approximation $\Et$ equals $E$ and $\mu Y$ equals $\curl E$. 
\end{theo}

\begin{rem}
\mylabel{theotworemone}
The assertions of the latter theorem remain valid 
if we replace $b_{\Et,\rho;Y,\gamma}$ by $B_{\Et,\rho;Y,\gamma}$.
\end{rem}

\begin{rem}
\mylabel{theotworemtwo}
The latter theorem provides a new variational formulation 
for the second order problem \eqref{maxeqcurlcurl}-\eqref{maxeqicptE}
and thus, in view of Remark \ref{secondtofirstorder},
for the original first order problem \eqref{maxeqone}-\eqref{maxeqictwo} as well.
\end{rem}

\subsection{Refinement of the estimate}

We can derive sharper estimates if $\rho$ and $\gamma$ in Theorem \ref{theoone}
depend on time. Then, we replace 
$$\nrom{\Phi}(t),\quad\Nrom{\Phi}(t)=\int_{0}^{t}\nrom{\Phi}(s)ds
=\int_{0}^{t}\big(\normepsom{\pt\Phi}^2(s)+\rho\normmumoom{\curl\Phi}^2(s)\big)ds$$
by
\begin{align*}
\ntrom{\Phi}(t)&:=\normepsom{\pt\Phi}^2(t)+\rho(t)\normmumoom{\curl\Phi}^2(t),\\
\Ntrgom{\Phi}(t)&:=\int_{0}^{t}\gamma(s)\ntrom{\Phi}(s)ds
=\int_{0}^{t}\gamma(s)\big(\normepsom{\pt\Phi}^2(s)+\rho(s)\normmumoom{\curl\Phi}^2(s)\big)ds,
\end{align*}  
respectively. In this case, $\Ntrgom{\Phi}'=\gamma\ntrom{\Phi}$
and we modify \eqref{estSl} in an obvious manner, i.e.,
\begin{align}
\begin{split}
2|S_{1}(t)|&\leq\alpha(t)\normmumoom{\curl e}^2(t)
+\alpha^{-1}(t)\normmuom{\tilde{K}_{\Et,Y}}^2(t),\\
2|S_{2}(t)|&\leq\int_{0}^{t}\beta(s)\normmumoom{\curl e}^2(s)ds
+\int_{0}^{t}\beta^{-1}(s)\normmuom{\pt \tilde{K}_{\Et,Y}}^2(s)ds,\\
2|S_{3}(t)|&\leq\int_{0}^{t}\gamma(s)\normepsom{\pt e}^2(s)ds
+\int_{0}^{t}\gamma^{-1}(s)\normepsmoom{\hat{K}_{\Et,Y}}^2(s)ds.
\end{split}\mylabel{estSlmod}
\end{align}
By \eqref{estnone} and \eqref{estSlmod} we find that
$$\ntrom{e}(t)\leq\int_{0}^{t}\gamma(s)\ntrom{e}(s)ds+\tilde{f}_{\Et,Y,\gamma,\rho}(t),\quad
\Ntrgom{e}'(t)\leq\gamma(t)\Ntrgom{e}(t)+\gamma(t)\tilde{f}_{\Et,Y,\gamma,\rho}(t),$$
where $\tilde{f}_{\Et,Y,\gamma,\rho}:=\tilde{g}_{\Et,Y,\gamma,\rho}+z_{\Et,Y}$ with
\begin{align*}
\tilde{g}_{\Et,Y,\gamma,\rho}(t)
&:=(1-\rho)^{-1}(t)\normmuom{\tilde{K}_{\Et,Y}}^2(t)\\
&\qquad+\int_{0}^{t}\gamma^{-1}(s)\normepsmoom{\hat{K}_{\Et,Y}}^2(s)ds
+\int_{0}^{t}(\gamma\rho)^{-1}(s)\normmuom{\pt \tilde{K}_{\Et,Y}}^2(s)ds.
\end{align*}
We apply \eqref{gronwalloneestone} and \eqref{gronwalltwoestone}, respectively,
and arrive at the following result:

\begin{theo}
\mylabel{theothree}
Let $\rho:\ol{I}\to(0,1)$ and $\Et$ be an approximation satisfying \eqref{secordreget}.
Moreover, let $\pt e\in\Hccom{}$ for all $t\in I$. Then, for all $t\in\ol{I}$
\begin{align}
\mylabel{theothreemain}
\ntrom{e}(t)\leq\inf_{Y,\gamma}\tilde{b}_{\Et,\rho;Y,\gamma}(t),\quad
\Ntrgom{e}(t)\leq\inf_{Y,\gamma}\tilde{B}_{\Et,\rho;Y,\gamma}(t),
\end{align}
where
\begin{align*}
\tilde{b}_{\Et,\rho;Y,\gamma}(t)
&:=\e^{\Gamma(t)}\int_{0}^{t}\e^{-\Gamma(s)}\gamma(s)\tilde{f}_{\Et,Y,\gamma,\rho}(s)ds
+\tilde{f}_{\Et,Y,\gamma,\rho}(t),\quad
\Gamma(t):=\int_{0}^{t}\gamma(s)ds,\\
\tilde{B}_{\Et,\rho;Y,\gamma}(t)
&:=\e^{\Gamma(t)}\int_{0}^{t}\e^{-\Gamma(s)}\gamma(s)\tilde{f}_{\Et,Y,\gamma,\rho}(s)ds
\end{align*}
and the infima are taken over $\gamma:\ol{I}\to\rzp$ and $Y\in\Cop(\ol{I},\Hcom{})$.
\end{theo}

\begin{rem}
\mylabel{theothreerem}
\begin{itemize}
\item[\bf(i)] The corresponding other assertions hold like in Theorem \ref{theoone}.
\item[\bf(ii)] If $\gamma$ is constant we get the same formulas as in Theorem \ref{theoone}, i.e.,
\begin{align*}
\tilde{b}_{\Et,\rho;Y,\gamma}(t)
&=\gamma\e^{\gamma t}\int_{0}^{t}\e^{-\gamma s}\tilde{f}_{\Et,Y,\gamma,\rho}(s)ds
+\tilde{f}_{\Et,Y,\gamma,\rho}(t),\\
\tilde{B}_{\Et,\rho;Y,\gamma}(t)
&=\gamma\e^{\gamma t}\int_{0}^{t}\e^{-\gamma s}\tilde{f}_{\Et,Y,\gamma,\rho}(s)ds.
\end{align*}
We note that in this case $\Ntrgom{e}=\gamma\Ntroom{e}$.
If $\gamma$ and $\rho$ are both constant the estimates coincide 
with those of Theorem \ref{theoone}.
\item[\bf(iii)] Since the latter estimates are stronger it is clear that
Theorem \ref{theotwo} and Remarks \ref{theotworemone} and \ref{theotworemtwo}
hold as well. 
\end{itemize}
\end{rem}

\section{Second form of the deviation majorant}\mylabel{secsecondest}

The estimates presented in Theorems \ref{theoone} and \ref{theothree} are derived 
for approximations $\Et$ having second order time derivatives.
This requirement may be difficult to satisfy in practice 
because typical approximate solutions possess only first order time derivatives. 
In this section, we derive estimates applicable for approximations of such a type.

As above, $\Et$ is an approximation of $E$, 
but now we also introduce a vector field $\Ett$ considered as an approximation of $\pt E$. 
Hence, we define both the error and the error of the time derivative separately by
$$e:=E-\Et,\quad\et:=\pt E-\Ett.$$ 
We note that in general $\Ett\neq\pt\Et$ and therefore $e_{t}\neq\pt e$. 
Henceforth, we assume that
\begin{align}
\mylabel{Etass}
\begin{split}
\Et,\Ett&\in\Cop(\ol{I},\Hcom{}),\\
\et&\in\Hccom{}\text{ for all }t\in I,
\end{split}
\intertext{where it would be sufficient to assume that}
\begin{split}
\Et&\in\Cop(\ol{I},\Hcom{}),\\
\Ett&\in\Cop(\ol{I},\Hom{})\cap\Lt(I,\Hcom{}),\\
\Ett&\in\Hccom{}\text{ for all }t\in I.
\end{split}\nonumber
\end{align}
With two nonnegative, real functions $\rho$ and $\gamma$ on $\ol{I}$
we define two energy norms
\begin{align*}
\nrom{\Phi,\Psi}(t)&:=\normepsom{\Phi}^2(t)+\rho(t)\normmumoom{\curl\Psi}^2(t),\\
\Nrgom{\Phi,\Psi}(t)&:=\int_{0}^{t}\gamma(s)\nrom{\Phi,\Psi}(s)ds
=\int_{0}^{t}\gamma(s)\big(\normepsom{\Phi}^2(s)+\rho(s)\normmumoom{\Psi}^2(s)\big)ds.
\end{align*}  
Then, $\Nrgom{\Phi,\Psi}'=\gamma\nrom{\Phi,\Psi}$.

\begin{theo}
\mylabel{theofour}
Let $\rho:\ol{I}\to(0,1)$ and $\Et$ be an approximation satisfying \eqref{Etass}.
Then, for all $t\in\ol{I}$
\begin{align}
\mylabel{theofourmain}
\nrom{\et,e}(t)\leq\inf_{Y,\gamma}b_{\Et,\Ett,\rho;Y,\gamma}(t),\quad
\Nrgom{\et,e}(t)\leq\inf_{Y,\gamma}B_{\Et,\Ett,\rho;Y,\gamma}(t),
\end{align}
where
\begin{align*}
b_{\Et,\Ett,\rho;Y,\gamma}(t)
&:=\e^{\Gamma(t)}\int_{0}^{t}\e^{-\Gamma(s)}\gamma(s)f_{\Et,\Ett,Y,\gamma,\rho}(s)ds
+f_{\Et,\Ett,Y,\gamma,\rho}(t),\quad
\Gamma(t):=\int_{0}^{t}\gamma(s)ds,\\
B_{\Et,\Ett,\rho;Y,\gamma}(t)
&:=\e^{\Gamma(t)}\int_{0}^{t}\e^{-\Gamma(s)}\gamma(s)f_{\Et,\Ett,Y,\gamma,\rho}(s)ds
\end{align*}
and the infima are taken over $\gamma:\ol{I}\to\rzp$ and $Y\in\Cop(\ol{I},\Hcom{})$.
Here,
\begin{align*}
f_{\Et,\Ett,Y,\gamma,\rho}&:=g_{\Et,\Ett,Y,\gamma,\rho}+z_{\Et,\Ett,Y},\\
g_{\Et,\Ett,Y,\gamma,\rho}(t)
&:=(1-\rho)^{-1}(t)\normmuom{\tilde{K}_{\Et,Y}}^2(t)+2\scpomt{\tilde{K}_{\Et,Y}}{\curl(\Ett-\pt\Et)}\\
&\qquad+\int_{0}^{t}\gamma^{-1}(s)\normepsmoom{\check{K}_{\Ett,Y}}^2(s)ds\\
&\qquad+\int_{0}^{t}(\gamma\rho)^{-1}(s)\normmuom{\mu^{-1}\curl\Ett-\pt Y}^2(s)ds,\\
z_{\Et,\Ett,Y}&:=\noom{\et,e}(0)+2\scpom{\tilde{K}_{\Et,Y}}{\curl e}(0),
\end{align*}
where $\check{K}_{\Ett,Y}:=\hat{K}_{\pt\Ett,\curl Y}:=\eps\pt\Ett+\curl Y-K$.
\end{theo}

\begin{rem}
\mylabel{theofourremone}
If $\Ett=\pt\Et$ then the estimates coincide with those of Theorem \ref{theothree}.
Furthermore, Remark \ref{theooneremone} holds in a similar way.
Particularly, $Y$ may be chosen from the larger space 
$$\Cop(\ol{I},\Hom{})\cap\Lt(I,\Hcom{}).$$
If $\gamma>0$ is constant then 
$$\Nrgom{\et,e}(t)=\gamma\Nroom{\et,e}(t)
=\gamma\int_{0}^{t}\nrom{\et,e}(s)ds,\quad
\Gamma(t)=\gamma t$$
and the upper bounds simplyfy, i.e.,
\begin{align*}
b_{\Et,\Ett,\rho;Y,\gamma}(t)
&=\gamma\e^{\gamma t}\int_{0}^{t}\e^{-\gamma s}f_{\Et,\Ett,Y,\gamma,\rho}(s)ds
+f_{\Et,\Ett,Y,\gamma,\rho}(t),\\
B_{\Et,\Et,\rho;Y,\gamma}(t)
&=\gamma\e^{\gamma t}\int_{0}^{t}\e^{-\gamma s}f_{\Et,\Ett,Y,\gamma,\rho}(s)ds.
\end{align*}
\end{rem}

If both $\gamma$ and $\rho$ are constant we have
$$\Nrgom{\Phi,\Psi}=\gamma\Nrom{\Phi,\Psi},\quad
\Nrom{\Phi,\Psi}:=\Nroom{\Phi,\Psi}.$$
In this case,
\begin{align*}
\nrom{\Phi,\Psi}(t)
&=\normepsom{\Phi}^2(t)+\rho\normmumoom{\curl\Psi}^2(t),\\
\Nrom{\Phi,\Psi}(t)
&=\normepsomt{\Phi}^2+\rho\normmumoomt{\curl\Psi}^2
\end{align*}
and
$$\nrom{\Phi,\Psi}=\Nrom{\Phi,\Psi}',\quad\Nrom{\Phi,\Psi}(t)
=\int_{0}^{t}\nrom{\Phi,\Psi}(s)ds.$$

\begin{theo}
\mylabel{theofive}
Let $\rho\in(0,1)$ and $\Et$ be an approximation satisfying \eqref{Etass}.
Then, for all $t\in\ol{I}$
\begin{align}
\mylabel{theofivemain}
\nrom{\et,e}(t)\leq\inf_{Y,\gamma}b_{\Et,\Ett\rho;Y,\gamma}(t),\quad
\Nrom{\et,e}(t)\leq\inf_{Y,\gamma}\hat{B}_{\Et,\Ett,\rho;Y,\gamma}(t),
\end{align}
where
\begin{align*}
b_{\Et,\Ett,\rho;Y,\gamma}(t)
&:=\gamma\e^{\gamma t}\int_{0}^{t}\e^{-\gamma s}f_{\Et,\Ett,Y,\gamma,\rho}(s)ds
+f_{\Et,\Ett,Y,\gamma,\rho}(t),\\
\hat{B}_{\Et,\Ett,\rho;Y,\gamma}(t)
&:=\e^{\gamma t}\int_{0}^{t}\e^{-\gamma s}f_{\Et,\Ett,Y,\gamma,\rho}(s)ds
\end{align*}
and the infima are taken over $\gamma\in\rzp$ and $Y\in\Cop(\ol{I},\Hcom{})$.
Here,
\begin{align*}
f_{\Et,\Ett,Y,\gamma,\rho}&:=g_{\Et,\Ett,Y,\gamma,\rho}+z_{\Et,\Ett,Y},\\
g_{\Et,\Ett,Y,\gamma,\rho}(t)
&:=(1-\rho)^{-1}\normmuom{\tilde{K}_{\Et,Y}}^2(t)+2\scpomt{\tilde{K}_{\Et,Y}}{\curl(\Ett-\pt\Et)}\\
&\qquad+\gamma^{-1}\normepsmoomt{\check{K}_{\Ett,Y}}^2
+(\gamma\rho)^{-1}\normmuomt{\mu^{-1}\curl\Ett-\pt Y}^2.
\end{align*}
\end{theo}

\begin{rem}
\mylabel{theofiveremone}
If $\Ett=\pt\Et$ then the estimates coincide with those of Theorem \ref{theoone}.
Again, Remark \ref{theooneremone} holds in a similar way.
In particular, $Y$ can be chosen from
$$\Cop(\ol{I},\Hom{})\cap\Lt(I,\Hcom{}).$$
\end{rem}

\begin{rem}
\mylabel{theofiveremtwo}
The absolute value of the zero term
$$z_{\Et,\Ett,Y}=\normepsom{\et}^2(0)+\normmumoom{\curl e}^2(0)+2\scpom{\tilde{K}_{\Et,Y}}{\curl e}(0)$$ 
can be estimated from above by the two quantities
\begin{align*}
\tilde{z}_{\Et,\Ett,Y}
&:=\noom{\et,e}(0)+2|\scpom{\tilde{K}_{\Et,Y}}{\curl e}(0)|\\
&\,\,=\normepsom{\et}^2(0)+\normmumoom{\curl e}^2(0)+2|\scpom{\tilde{K}_{\Et,Y}}{\curl e}|(0)\\
\leq\hat{z}_{\Et,\Ett,Y}
&:=\normepsom{\et}^2(0)+2\normmumoom{\curl e}^2(0)+\normmuom{\tilde{K}_{\Et,Y}}^2(0),
\end{align*}
which are nonnegative and easily computable.
The same manipulation can be done with the term 
$2\scpomt{\tilde{K}_{\Et,Y}}{\curl(\Ett-\pt\Et)}$,
taking, e.g., it's absolute value, 
which leads to some nonnegative $\tilde{g}_{\Et,\Ett,Y,\gamma,\rho}$.
Moreover,
\begin{itemize}
\item[\bf(i)] $z_{\Et,\Ett,Y}=0$, 
if $\et(0)=0$ and $\curl e(0)=0$.
\item[\bf(ii)] $\tilde{z}_{\Et,\Ett,Y}=0$, 
if any only if $\et(0)=0$ and $\curl e(0)=0$.
\item[\bf(iii)] $\hat{z}_{\Et,\Ett,Y}=0$, 
if any only if $\et(0)=0$ and $\curl e(0)=0$ and $\mu Y(0)=\curl\Et(0)$.
\item[\bf(iv)] $\tilde{g}_{\Et,\Ett,Y,\gamma,\rho}=0$, 
if and only if $\mu Y=\curl\Et$, $\pt\mu Y=\curl\Ett$ and $\curl Y=K-\eps\pt\Ett$.
\end{itemize}
Therefore, choosing the functional $f_{\Et,\Ett,Y,\gamma,\rho}$ 
with $\tilde{z}_{\Et,Y}$ 
or $\hat{z}_{\Et,Y}$ 
and $\tilde{g}_{\Et,\Ett,Y,\gamma,\rho}$ we see that
for all $t\in\ol{I}$
the functional $b_{\Et,\Ett,\rho;Y,\gamma}(t)$ vanishes, 
if and only if
\begin{align}
\Ett(0)&=E_{0}',&\curl\Et(0)&=\curl E_{0},\mylabel{bvanishtwo}\\
\mu Y&=\curl\Et,&\pt\mu Y&=\curl\Ett,&\curl Y&=K-\eps\pt\Ett.\mylabel{bvanishthree}
\end{align}
Thus, $\et$ and $\curl e$ vanish, if and only if $\nrom{\et,e}=0$,
which is implied by $b_{\Et,\Ett,\rho;Y,\gamma}=0$.
The latter constraint is equivalent to \eqref{bvanishtwo} and \eqref{bvanishthree}.
The same holds true for the energy norms $\Nrom{\et,e}$, $\Nrgom{\et,e}$ 
and the functionals $B_{\Et,\Ett,\rho;Y,\gamma}$, $\hat{B}_{\Et,\Ett,\rho;Y,\gamma}$.
\end{rem}

\begin{proofof}{Theorem \ref{theofour}}
We follow in close lines the proofs of Theorems \ref{theoone} and \ref{theothree}.
Since $\et\in\Hccom{}$ and $\pt e=\et+\Ett-\pt\Et$, we have  
\begin{align*}
\pt\noom{\et,e}(t)&=2\scpom{\eps\pt\et}{\et}(t)
+2\scpom{\mu^{-1}\curl e-Y+Y}{\curl\et}(t)\\
&\qquad+2\scpom{\mu^{-1}\curl e}{\curl(\Ett-\pt\Et)}(t)\\
&=2\scpom{K-\curl Y-\eps\pt\Ett}{\et}(t)-2\scpom{\tilde{K}_{\Et,Y}}{\curl\et}(t)\\
&\qquad+2\scpom{\curl(\Ett-\pt\Et)}{\mu^{-1}\curl e}(t)\\
&=-2\scpom{\check{K}_{\Ett,Y}}{\et}(t)-2\pt\scpom{\tilde{K}_{\Et,Y}}{\curl e}(t)\\
&\qquad+2\scpom{\tilde{K}_{\Et,Y}}{\curl(\Ett-\pt\Et)}(t)\\
&\qquad+2\scpom{\ub{\mu^{-1}\curl(\Ett-\pt\Et)
+\pt \tilde{K}_{\Et,Y}}_{=\mu^{-1}\curl\Ett-\pt Y}}{\curl e}(t).
\end{align*}
Thus, by integration
\begin{align}
&\qquad\noom{\et,e}(t)\nonumber\\
&=z_{\Et,\Ett,Y}+2\scpomt{\tilde{K}_{\Et,Y}}{\curl(\Ett-\pt\Et)}\mylabel{estntwo}\\
&\qquad-2\big(\ub{\scpom{\tilde{K}_{\Et,Y}}{\curl e}}_{=:S_{1}}(t)
-\ub{\scpomt{\mu^{-1}\curl\Ett-\pt Y}{\curl e}}_{=:S_{2}(t)}
+\ub{\scpomt{\check{K}_{\Ett,Y}}{\et}}_{=:S_{3}(t)}\big).\nonumber
\end{align}
If $\Ett=\pt\Et$ then \eqref{estntwo} coincides with \eqref{estnone}.
As before, we choose $\alpha:=1-\rho$ and $\beta:=\gamma\rho$
and estimate the scalar products $S_{\ell}$ as follows: 
\begin{align}
\begin{split}
2|S_{1}(t)|&\leq\alpha(t)\normmumoom{\curl e}^2(t)
+\alpha^{-1}(t)\normmuom{\tilde{K}_{\Et,Y}}^2(t)\\
2|S_{2}(t)|&\leq\int_{0}^{t}\beta(s)\normmumoom{\curl e}^2(s)ds
+\int_{0}^{t}\beta^{-1}(s)\normmuom{\mu^{-1}\curl\Ett-\pt Y}^2(s)ds\\
2|S_{3}(t)|&\leq\int_{0}^{t}\gamma(s)\normepsom{\et}^2(s)ds
+\int_{0}^{t}\gamma^{-1}(s)\normepsmoom{\check{K}_{\Ett,Y}}^2(s)ds
\end{split}\mylabel{estSltwo}
\end{align}
Inserting \eqref{estSltwo} into \eqref{estntwo} yields
\begin{align*}
\nrom{\et,e}(t)&\leq\int_{0}^{t}\gamma(s)\nrom{\et,e}(s)ds+f_{\Et,\Ett,Y,\gamma,\rho}(t),\\
\Nrgom{\et,e}'(t)&\leq\gamma(t)\Nrgom{\et,e}(t)+\gamma(t)f_{\Et,\Ett,Y,\gamma,\rho}(t).
\end{align*}
Finally, Gronwall's inequalities, 
i.e., \eqref{gronwalloneestone} and \eqref{gronwalltwoestone}, 
prove the assertions.
\end{proofof}

Since $\et=0$ implies 
$$e(t)=\int_{0}^{t}(\Ett-\pt\Et)(s)ds+e(0)$$
we obtain:

\begin{theo}
\mylabel{theosix}
Let approximations $\Et$, $\Ett$ and $Y$ as in Theorem \ref{theofour}
or Theorem \ref{theofive} be given.
Then, the following two statements are equivalent:
\begin{itemize}
\item[\bf(i)] $\Et(0)=E_{0}$ and $\Ett=\pt\Et$ 
and $b_{\Et,\Ett,\rho;Y,\gamma}=0$.
\item[\bf(ii)] $\Et=E$ and $\Ett=\pt E$ and $\mu Y=\curl E$.
\end{itemize}
In words: Let the approximations $\Et$, $\Ett$ satisfy $\Ett=\pt\Et$
and the first initial condition $\Et(0)=E_{0}$ exactly. Then, the functional 
$b_{\Et,\Ett,\rho;Y,\gamma}$ vanishes if and only if 
the approximation $\Et$ equals $E$ and $\mu Y$ equals $\curl E$. 
\end{theo}

\begin{rem}
\mylabel{theosixremone}
The assertions of the latter theorem remain valid
if we replace $b_{\Et,\Ett,\rho;Y,\gamma}$ 
by $B_{\Et,\Ett,\rho;Y,\gamma}$ or $\hat{B}_{\Et,\Ett,\rho;Y,\gamma}$. 
\end{rem}

\begin{rem}
\mylabel{theosixremtwo}
The latter theorems provide new variational formulations 
for the second order problem \eqref{maxeqcurlcurl}-\eqref{maxeqicptE}
and thus, in view of Remark \ref{secondtofirstorder},
for the original first order problem \eqref{maxeqone}-\eqref{maxeqictwo} as well.
\end{rem}

\section{Estimates for the approximation of the whole solution}\mylabel{bothfields}

By \eqref{maxeq} (or the basic equations \eqref{maxeqone}, \eqref{maxeqtwo})
we also get estimates for the errors $h$, $h_{t}$ of the magnetic fields 
$H$, $\p_{t}H$ and their approximations $\tilde{H}$, $\tilde{H}_{t}$.
E.g., by adding 
$-(\Ett,\tilde{H}_{t})+\iu M_{\Lambda}(\Et,\tilde{H})$ to \eqref{maxeq} we obtain
$$(\et,h_{t})-\iu M_{\Lambda}(e,h)=(f,g)
:=(F,G)-(\Ett,\tilde{H}_{t})+\iu M_{\Lambda}(\Et,\tilde{H}),$$
which reads explicitly 
\begin{align*}
\et-\eps^{-1}\curl h&=f=F-\tilde{E}_{t}+\eps^{-1}\curl\tilde{H},\\
h_{t}+\mu^{-1}\curl e&=g=G-\tilde{H}_{t}-\mu^{-1}\curl\Et.
\end{align*}
Therefore, we can estimate
\begin{align*}
\hat{n}_{h_{t},h,\rho}(t)
&:=\rho(t)\norm{h_{t}}^2_{\mu,\Omega}(t)+\norm{\curl h}^2_{\eps^{-1},\Omega}(t)\\
&\,\,\leq2\nrom{\et,e}(t)
+2\norm{f}^2_{\eps,\Omega}(t)+2\normmuom{g}^2(t)\\
&\,\,\leq2\inf_{Y,\gamma}b_{\Et,\Ett,\rho;Y,\gamma}(t)
+2\norm{f}^2_{\eps,\Omega}(t)+2\normmuom{g}^2(t),
\end{align*}
which yields
\begin{align*}
&\qquad\nrom{\et,e}(t)+\hat{n}_{h_{t},h,\rho}(t)\\
&=\norm{e_{t}}^2_{\eps,\Omega}(t)+\rho(t)\norm{\curl e}^2_{\mu^{-1},\Omega}(t)
+\rho(t)\norm{h_{t}}^2_{\mu,\Omega}(t)+\norm{\curl h}^2_{\eps^{-1},\Omega}(t)\\
&\leq3\inf_{Y,\gamma}b_{\Et,\Ett,\rho;Y,\gamma}(t)
+2\norm{f}^2_{\eps,\Omega}(t)+2\normmuom{g}^2(t).
\end{align*}
Of course, similar estimates hold for the other norms and functionals
and the estimates simplify in an obvious way if $\rho$ or $\gamma$ are positive constants.
The vector fields $(f,g)$ measure the error in the original first order equation \eqref{maxeq}.
Moreover, $(\et,h_{t})$ may be replaced by $\p_{t}(e,h)$ if for the approximations 
sufficient regularity is available. In this case, the error in the first order equation is
$$(f,g)=(F,G)-(\p_{t}-\iu M_{\Lambda})(\Et,\tilde{H})=(\p_{t}-\iu M_{\Lambda})(e,h).$$

\appendix

\section{Appendix: Gronwall inequalities}\mylabel{secappendix}

Gronwall inequalities can be found in almost any book about ordinary differential equations.
Since these estimates differ in small details, we present here two estimates, which meet our needs.
With
$$\Cop(\ol{I}):=\Cop(\ol{I},\rz),\quad\Cz(\ol{I}):=\Cz(\ol{I},\rz)$$
we have

\begin{lem}
\mylabel{gronwallone}
{\sf(differential form)}
Let $u\in\Cop(\ol{I})$ and $\varphi,\psi\in\Cz(\ol{I})$ 
with $\varphi\geq0$. If in $I$
$$u'\leq\varphi u+\psi$$
then for all $t\in\ol{I}$
\begin{align}
\mylabel{gronwalloneestone}
u(t)&\leq\exp(\Phi(t))\big(u(0)+\int_{0}^{t}\exp(-\Phi(s))\psi(s)ds\big),\quad
\Phi(t):=\int_{0}^{t}\varphi(s)ds.
\intertext{If $\varphi$ is a nonnegative constant then for all $t\in\ol{I}$}
\mylabel{gronwalloneesttwo}
u(t)&\leq\exp(\varphi t)\big(u(0)+\int_{0}^{t}\exp(-\varphi s)\psi(s)ds\big).
\intertext{If $\psi\leq c\in\rz$ then for all $t\in\ol{I}$}
\nonumber
u(t)&\leq(u(0)+ct)\exp(\Phi(t)).
\end{align}
\end{lem}

\begin{proof}
Since
$$(\exp(-\Phi)u)'=\exp(-\Phi)(u'-\varphi u)\leq\exp(-\Phi)\psi$$
we have 
$$\exp(-\Phi(t))u(t)\leq u(0)+\int_{0}^{t}\exp(-\Phi(s))\psi(s)ds,$$
which proves the first part. The other assertions are trivial.
\end{proof}

Using the notations of the latter lemma, we have

\begin{lem}
\mylabel{gronwalltwo}
{\sf(integral form)}
Let $u,\varphi,\psi\in\Cz(\ol{I})$ with $\varphi\geq0$. If for all $t\in\ol{I}$
$$u(t)\leq\int_{0}^{t}\varphi(s)u(s)ds+\psi(t)$$
then for all $t\in\ol{I}$
\begin{align}
\mylabel{gronwalltwoestone}
u(t)&\leq\exp(\Phi(t))\int_{0}^{t}\exp(-\Phi(s))\varphi(s)\psi(s)ds+\psi(t).
\intertext{If $\varphi$ is a nonnegative constant then for all $t\in\ol{I}$}
\mylabel{gronwalltwoesttwo}
u(t)&\leq\varphi\exp(\varphi t)\int_{0}^{t}\exp(-\varphi s)\psi(s)ds+\psi(t).
\intertext{If $\psi\leq c\in\rz$ then for all $t\in\ol{I}$}
\nonumber
u(t)&\leq c\exp(\Phi(t)).
\end{align}
\end{lem}

\begin{proof}
Set $\displaystyle\tilde{u}(t):=\int_{0}^{t}\varphi(s)u(s)ds$ and $\tilde{\psi}:=\varphi\psi$. 
Then $\tilde{u}'=\varphi u\leq\varphi\tilde{u}+\tilde{\psi}$.
Lemma \ref{gronwallone} yields
$$u(t)\leq\tilde{u}(t)+\psi(t)
\leq\exp(\Phi(t))\big(\tilde{u}(0)+\int_{0}^{t}\exp(-\Phi(s))\tilde{\psi}(s)ds\big)+\psi(t),$$
which proves the first assertion since $\tilde{u}(0)=0$. The second assertion is trivial.
To prove the last part we compute with $\exp(-\Phi)'=-\exp(-\Phi)\varphi$
$$u(t)\leq c\exp(\Phi(t))\int_{0}^{t}\exp(-\Phi(s))\varphi(s)ds+c=c\exp(\Phi(t)).$$
\end{proof}

\begin{rem}
\mylabel{gronwallrem}
The differential and integral form are equivalent.
\end{rem}
 
\begin{proof}
Let the assumptions of Lemma \ref{gronwallone} be satisfied.
Then, 
$$u(t)\leq\int_{0}^{t}\varphi(s)u(s)ds+\ub{\int_{0}^{t}\psi(t)ds+u(0)}_{=:\tilde{\psi}(t)}$$
and by Lemma \ref{gronwalltwo}
\begin{align*}
u(t)&\leq\exp(\Phi(t))\int_{0}^{t}\ub{\exp(-\Phi(s))\varphi(s)}_{\uparrow}
\ub{\tilde{\psi}(s)}_{\downarrow}ds+\tilde{\psi}(t)\\
&=\exp(\Phi(t))\big(\int_{0}^{t}\exp(-\Phi(s))\ub{\tilde{\psi}'(s)}_{=\psi(s)}ds
-\ub{\exp(-\Phi(s))\tilde{\psi}(s)\Big|_{0}^{t}}_{=\exp(-\Phi(t))\tilde{\psi}(t)-\tilde{\psi}(0)}\big)
+\tilde{\psi}(t),
\end{align*}
which completes the proof since $\tilde{\psi}(0)=u(0)$.
\end{proof}

\begin{cor}
\mylabel{gronwallcor}
\begin{itemize}
\item[\bf(i)] Let the assumptions of Lemma \ref{gronwallone} by satisfied. 
Then, $u(0)\leq0$ and $\psi\leq0$ imply $u\leq0$. 
Hence, if $u\geq0$ then $u(0)\leq0$ and $\psi\leq0$ imply $u=0$.
\item[\bf(ii)] Let the assumptions of Lemma \ref{gronwalltwo} by satisfied. 
Then, $\psi\leq0$ implies $u\leq0$. 
Hence, if $u\geq0$ then $\psi\leq0$ implies $u=0$.
\end{itemize}
\end{cor}

\begin{acknow}
The research was supported by the University of Jyv\"askyl\"a and the Academy of Finland.
\end{acknow}

{}

\addressesdirksergeytuomo

\end{document}